\documentclass[preprint]{elsarticle}
\usepackage{amsmath, amsthm,amsfonts,amssymb,graphicx}
\usepackage{xcolor}
\usepackage{times,enumerate}
\usepackage{array}
\newcolumntype{C}[1]{>{\centering\arraybackslash$}p{#1}<{$}}

\newcommand{\cH}{{\cal H}}

\newcommand{\cS}{{\cal S}}
\newcommand{\cT}{{\cal T}}

\newcommand{\ZZ}{\mathbb{Z}}
\newcommand{\RR}{\mathbb{R}}
\newcommand{\NN}{\mathbb{N}}
\newcommand{\CC}{\mathbb{C}}
\newcommand{\TT}{\mathbb{T}}

\newcommand{\Ab}{{\boldsymbol{A}}}
\newcommand{\Bb}{{\boldsymbol{B}}}

\newcommand{\Db}{{\boldsymbol{D}}}

\newcommand{\Fb}{{\boldsymbol{F}}}

\newcommand{\Hb}{{\boldsymbol{H}}}
\newcommand{\Ib}{{\boldsymbol{I}}}

\newcommand{\Kb}{{\boldsymbol{K}}}

\newcommand{\Nb}{{\boldsymbol{N}}}

\newcommand{\Rb}{{\boldsymbol{R}}}
\newcommand{\Sb}{{\boldsymbol{S}}}
\newcommand{\Tb}{{\boldsymbol{T}}}

\newcommand{\cb}{{\boldsymbol{c}}}
\newcommand{\db}{{\boldsymbol{d}}}
\newcommand{\eb}{{\boldsymbol{e}}}
\newcommand{\fb}{{\boldsymbol{f}}}

\newcommand{\hb}{{\boldsymbol{h}}}

\newcommand{\vb}{{\boldsymbol{v}}}

\newcommand{\yb}{{\boldsymbol{y}}}

\newcommand{\bPhi}{{\boldsymbol{\Phi}}}
\newcommand{\bPsi}{{\boldsymbol{\Psi}}}

\newcommand\bphi{{\boldsymbol \phi}}
\newcommand\bZero{{\boldsymbol 0}}

\newcommand{\tbPhi}{{\widetilde{\boldsymbol{\Phi}}}}
\newcommand{\tbPsi}{\widetilde{\boldsymbol{\Psi}}}

\newcommand{\tAb}{\widetilde{{\boldsymbol{A}}}}
\newcommand{\tBb}{\widetilde{{\boldsymbol{B}}}}

\newcommand{\supp}{{\mathop{\mbox{\rm supp}\,}}}
\newcommand{\spann}{{\mathop{\mbox{\rm span}\,}}}
\newcommand{\tV}{{\widetilde{V}}}

\newcommand{\ntAb}[1]{\tAb^{[#1]}} 
\newcommand{\nAb}[1]{\Ab^{[#1]}}
\newcommand{\ntBb}[1]{\tBb^{[#1]}} 
\newcommand{\nBb}[1]{\Bb^{[#1]}} 
\newcommand{\nDb}[1]{\Db^{#1}} 
\newcommand{\ncb}[1]{\cb^{[#1]}} 
\newcommand{\nHb}[1]{\Hb^{[#1]}} 
 
\newtheorem{theorem}{Theorem}
\newtheorem{proposition}[theorem]{Proposition}

\newtheorem{definition}[theorem]{Definition}

\begin{document}
\begin{frontmatter}	
\title{A note on Hermite multiwavelets with polynomial and exponential vanishing moments}

\author[MariantoAddress]{Mariantonia Cotronei\corref{mycorrespondingauthor}}
\cortext[mycorrespondingauthor]{Corresponding author}
\ead{mariantonia.cotronei@unirc.it}
\address[MariantoAddress]{DIIES, Universit\`a Mediterranea di Reggio Calabria, Via Graziella, 89122 Reggio Calabria,
	Italy}

\author[NadaAddress]{Nada Sissouno}
\ead{sissouno@forwiss.uni-passau.de}
\address[NadaAddress]{Chair of Digital Image Processing, University of Passau, Innstr. 43,
	94032 Passau, Germany}

\begin{abstract}
	The aim of the paper is to  present Hermite-type multiwavelets satisfying the
	vanishing moment property with respect to elements in the space spanned by exponentials and polynomials. Such functions satisfy a two-scale relation which is level-dependent as well as the corresponding  multiresolution analysis. An important feature of the associated  filters is the possibility of factorizing  their symbols in terms of the so-called cancellation operator. A family of biorthogonal multiwavelet system possessing the above property and obtained from a Hermite subdivision scheme reproducing  polynomial and exponential data is  finally introduced. 
\end{abstract}

\end{frontmatter}

\section{Introduction}

It is well-known that \emph{multiwavelets} generalize classical wavelets in the sense that the corresponding multiresolution analysis is generated by translates and dilates of not just one but several functions. These functions can be assembled in a   vector, also known as {\em multi-scaling function}, satisfying a {\em vector refinement equation}, whose coefficients are matrices rather than scalars (see \cite{Keinert} for an overview on the topic). Such generalization can result in some advantages connected to the possibility of constructing bases, for example, with  short support and  high approximation order. Nevertheless, the approximation order properties cannot be exploited directly in practical implementations, because they do not imply a corresponding discrete polynomial preservation/cancellation property on the filters side. This results in combining the discrete multiwavelet transform with computationally costly pre-processing and post-processing steps \cite{BCS,CLS}, unless full-rank filters \cite{CH,CCS2008,CCS2010} or balanced multiwavelets \cite{BCL,LebrunVetterli} are used. Also, except these cases, no easy factorization of the symbol as in the scalar situation can be considered.

This paper deals with multiwavelets of Hermite-type, connected with  multi-scaling function vectors  whose elements satisfy  Hermite conditions. 
In particular, we are interested in multiwavelet filters which provide not only polynomial but also exponential data cancellation.
We thus use a notion of vanishing moment which extends the one usually given, which refers just to polynomials. This generalized property  assures certain compression capabilities of the wavelet system also in the case where the given data exhibit  transcendental features. Wavelets possessing such  property have already been studied for example in \cite{UnserBlu05} in a scalar framework. The vector context offers the advantage of providing a higher number of vanishing moments together with  a  short support. Hermite-type multiwavelets allow, in addition, to express the cancellation property as the factorization  of the wavelet filter in terms of the so-called \emph{annihilator} or  {\em cancellation operator} introduced in \cite{ContiCotroneiSauer15}    in the context of the study  of \emph{Hermite subdivision schemes}. These  are level-dependent schemes acting on vector data representing function values and consecutive derivatives up to  a certain order (see, for example, \cite{CMR,CRU,DynLevin99,DubucMerrien06,MerrienSauer12}). In \cite{ContiCotroneiSauer15,ContiCotroneiSauer16} some conditions have been proved connected to the preservation of elements in the (polynomial and exponential) space spanned by
$\left\{1,x,\ldots,x^{p},e^{\pm\lambda_1 x},\cdots, e^{\pm\lambda_r x}\right\}$, with $p,r\in \NN$. In particular the preservation property allows the factorization of the subdivision operator in terms  of a minimal annihilator.

Our idea is to exploit the  close connection between subdivision schemes and wavelet analysis to study biorthogonal multiwavelet filters of Hermite type, in the sense that the underlying multi-scaling function is associated to Hermite subdivision schemes. 
In particular, 
we show how, given a Hermite subdivision operator based  on a level-dependent mask $\Ab^{[n]}$,  satisfying the 
$V_{d,\Lambda}$-spectral condition, in the sense specified later,
 it is always possible to complete it to a biorthogonal system, where the wavelet filter possess the desired polynomial/exponential cancellation property.
In particular, 
 we focus on a special construction of Hermite-type multiwavelet biorthogonal  systems, based on an  MRA realized from the interpolatory subdivision scheme provided in \cite{ContiCotroneiSauer16}. Such an MRA is generated by a level-dependent vector refinable function which turns out to be a generalization of the well-known Hermite (or finite element) 
  multi-scaling function proposed, for example, by Strang and Strela  in \cite{Strela1995}. 

The paper is organized as follows. In Section~\ref{sec:basics} we fix the notation and present some basic facts about  level-dependent (nonstationary) multiresolution analyses of $L^2(\RR)$ and related discrete wavelet transforms. In Section~\ref{sec:hermite} we provide some details and properties of Hermite subdivision schemes preserving exponential and polynomial data.  A construction of the Hermite multiwavelets from such schemes is proposed in Section~\ref{sec:MRA_fact}, and a factorization result is formulated. Finally,
in Section~\ref{sec:HermiteInterp} we give an example of our construction, 
based on an explicitly given family of Hermite subdivision possessing preservation properties. Some conclusions are drawn in Section~\ref{sec:conclusion}.

\section{Preliminaries and basic facts}\label{sec:basics}
Let $\ell^r(\ZZ)$ and $\ell^{r\times r}(\ZZ)$, respectively, denote the spaces of all vector-valued and matrix-valued sequences  defined  on $\ZZ$. 


A  \emph{level-dependent MRA} of $L^2(\RR)$ is defined as the nested sequence
$V_0\subset V_1\subset \cdots\subset L^2(\RR)$
of spaces each spanned by the dilates and translates of a finite set of functions, which differs from level to level, that is, for $d\in\NN$,
\begin{equation}	\label{eq:primal_MRA}
V_n
:=
\spann\{
\phi_0^{[n]}(2^n\cdot - k),\dots,\phi_d^{[n]}(2^n\cdot - k): k\in\ZZ
\},\quad n\in\NN.
\end{equation}
Nonstationary MRAs, in the scalar case ($d=0$), have been introduced, for example, in \cite{CohenDyn,PitolliACOM}.

For each $n\in \NN$, such functions can be arranged in a column vector
$
\bPhi^{[n]}
:=
[\phi_0^{[n]},\allowbreak \phi_1^{[n]},\allowbreak \dots,\phi_d^{[n]}]^T
$.
The dependency of two vector functions at different levels is given in terms of the
{\em level-dependent two-scale-relation}
\begin{equation}	\label{eq:primal_2_scale}
(\bPhi^{[n-1]})^T=\sum_{k\in\ZZ}(\bPhi^{[n]})^T(2\cdot - k)\nAb{n-1}_k,
\end{equation}
where the matrix-valued sequence $\nAb{n}:=(\nAb{n}_k:k\in\ZZ)\in\ell^{(d+1)\times (d+1)}(\ZZ)$
is called the {\em mask} of $\bPhi^{[n]}$. 

In a {\em biorthogonal}  setting those functions and spaces
play the role of the {\it primal scaling function vectors} and
{\em decomposition spaces}. From the point of view of filter banks the masks
correspond to the {\it low-pass filters} in the decomposition.

Given a second level-dependent MRA $
(\tV_n:n\in\NN)$ generated by $
\tbPhi^{[n]}$
satisfying
\begin{equation*}	
(\tbPhi^{[n-1]})^T=\sum_{k\in\ZZ}(\tbPhi^{[n]})^T(2\cdot - k)\ntAb{n-1}_k
\end{equation*}
for some matrix-valued masks $\ntAb{n}\in\ell^{(d+1)\times (d+1)}(\ZZ)$, then
the spaces $\tV_n$ represent the {\it reconstruction spaces} with
{\it dual scaling function vectors} $\tbPhi^{[n]}$ 
if the following duality relations are satisfied
\begin{equation} \label{eq:scaling_duality}
\langle \bPhi^{[n]},\tbPhi^{[n]}(\cdot+k)\rangle
:=
\int_{\RR}\bPhi^{[n]}(x)(\tbPhi^{[n]})^T(x+k)\,dx =\delta_{k,0}\Ib,
\quad k\in\ZZ.
\end{equation}

Let $W_n$ and $\widetilde{W}_n$ denote the
 \emph{wavelet spaces} at level $n$, that is, the complementary
subspaces of $V_n$ in $V_{n+1}$ and $\tV_n$ in $\tV_{n+1}$,
respectively. Those spaces are generated by the shifts of the components of the 
vector-valued functions
$
\bPsi^{[n]}$ and
$
\tbPsi^{[n]}
$.
 Since, by construction, $W_n\subset V_{n+1}$ and
$\widetilde{W}_n\subset\tV_{n+1}$ there exist two matrix-valued masks
$\nBb{n},\, \ntBb{n}\in\ell^{(d+1)\times (d+1)}(\ZZ)$ such that
\begin{equation*} 
\begin{array}{l}
(\bPsi^{[n]})^T
=
\displaystyle{\sum_{k\in\ZZ}(\bPhi^{[n+1]})^T(2\cdot - k)\nBb{n}_k},
\\
(\tbPsi^{[n]})^T
=
\displaystyle{\sum_{k\in\ZZ}(\tbPhi^{[n+1]})^T(2\cdot - k)\tBb_k^{[n]}.}
\end{array}
\end{equation*}
The masks $\nBb{n},\, \ntBb{n}$ correspond to
{\em high-pass filters} in the filter bank terminology. 

The function vectors $\bPsi^{[n]}$ and $\tbPsi^{[n]}$
represent the  level-dependent  {\em multiwavelets}
associated to the scaling functions
$\bPhi^{[n]}$ and $\tbPhi^{[n]}$ if they fulfill the following biorthogonality 
conditions: 
\begin{eqnarray}	\label{eq:wavelet_duality_0}
&&\langle \bPhi^{[n]},\tbPsi^{[n]}(\cdot+k)\rangle
= 
\langle \tbPhi^{[n]},\bPsi^{[n]}(\cdot+k)\rangle
 \; = \; \bZero,\\	
&&\langle \bPsi^{[n]},\tbPsi^{[n]}(\cdot+k)\rangle
= \delta_{k,0}\Ib,	\label{eq:wavelet_duality_I}
\end{eqnarray}
for $k\in\ZZ$, where $\bZero$ denotes the zero matrix.

For a finitely supported mask 
 $\nAb{n}\in\ell^{r\times r}(\ZZ)$, $r\in\NN$, the
{\em symbol} is defined as the matrix-valued Laurent polynomial
$\nAb{n}(z):=\sum_{k\in\ZZ}\nAb{n}_k\,z^k$,   $z\in\CC$. The  duality relations \eqref{eq:scaling_duality},
\eqref{eq:wavelet_duality_0}, and \eqref{eq:wavelet_duality_I}
can be expressed in terms of the
some conditions on the symbols of the masks on the unit circle $\TT:=\{z\in\CC:\,|z|=1\}$, namely
\begin{equation}\label{eq:biortcond}
\begin{array}{l}
(\ntAb{n})^\sharp(z)\,\nAb{n}(z)
	+(\ntAb{n})^\sharp(-z)\,\nAb{n}(-z) = 2\Ib,\\
(\ntAb{n})^\sharp(z)\,\nBb{n}(z)
	+(\ntAb{n})^\sharp(-z)\,\nBb{n}(-z) = \bZero,\\
(\ntBb{n})^\sharp(z)\,\nAb{n}(z)
	+(\ntBb{n})^\sharp(-z)\,\nAb{n}(-z) = \bZero,\\
(\ntBb{n})^\sharp(z)\,\nBb{n}(z)
	+(\ntBb{n})^\sharp(-z)\,\nBb{n}(-z) =2\Ib,
\end{array}
\end{equation}
where we have used the notation $(\Ab^{[n]})^\sharp(z):=(\Ab^{[n]})^T(z^{-1})$.


Suppose we are now given a function  $f\in V_n\subset L^2(\RR)$.
It can be represented as
$f=\sum_{k\in \ZZ} 
(\bPhi^{[n]})^T(2^{n}\cdot -k) \cb_k^{[n]}
$
for some coefficient sequence $(\cb_k^{[n]}:k\in\ZZ )\in \ell^{d+1}(\ZZ)$. Starting
from such sequence, a recursive  scheme can be derived for computing
all the coefficients involved in the decomposition
\begin{eqnarray*}
	f&=&{\cal P}_{n-1}f+{\cal Q}_{n-1}{f}=
{\cal P}_{n-2}f+{\cal Q}_{n-2}{f}+{\cal Q}_{n-1}{f}=\dots\\
&=&{\cal P}_{n-L}f+{\cal Q}_{n-L}f+{\cal Q}_{n-L+1}f+\dots +{\cal Q}_{n-1}f,
\end{eqnarray*}
where $L>0$ is fixed and ${\cal P}_{j}$, ${\cal Q}_{j}$ represent the projection operators
on the spaces $V_j$, $W_j$ respectively. In fact, for example,
\begin{eqnarray*}
{\cal P}_{n-1}{f}=\sum_{k\in \ZZ} 
(\bPhi^{[n-1]})^T(2^{n-1}\cdot -k) \cb_k^{[n-1]}
\end{eqnarray*}
with
\begin{eqnarray*}
  \cb_k^{[n-1]}&=&\langle (\tilde \bPhi^{[n-1]})^T(2^{n-1}\cdot -k),f\rangle\\
  &=&\langle \sum_{j\in \ZZ}(\tilde \bPhi^{[n]})^T(2^{n}\cdot -2k-j)
	  \tilde \Ab_k^{[n-1]},f\rangle\\
  &=& \sum_{j\in \ZZ}(\tilde \Ab_{j -2k}^{[n-1]})^T
	  \langle (\tilde \bPhi^{[n]})^T(2^{n}\cdot -j),f\rangle\\
  &=& \sum_{j\in \ZZ}(\tilde \Ab_{j -2k}^{[n-1]})^T \cb_j^{[n]}.
\end{eqnarray*}
The wavelet coefficients can be computed analogously. By recursively
applying the formulas, fixing $L\le n$, the \emph{decomposition scheme}
reads as
\begin{equation}\label{eq:dec_scheme}
 \left\{ \begin{array}{l}
  \displaystyle{\cb_k^{[n-\ell]} =
  \sum_{j\in \ZZ}(\tilde\Ab_{j -2k}^{[n-\ell]})^T \cb_j^{[n-\ell+1]}},\\\noalign{\medskip}
	\displaystyle{\db_k^{[n-\ell]} =
	\sum_{j\in \ZZ}(\tilde\Bb_{j -2k}^{[n-\ell]})^T
	\cb_j^{[n-\ell+1]}},\end{array}
	 \quad \ell=1,\dots,L.\right.
\end{equation}
By using similar arguments, one can derive the \emph{reconstruction scheme}
\begin{equation}\label{eq:rec_scheme}
\cb_k^{[n-\ell+1]}= \sum_{j\in \ZZ}\Ab_{k -2j}^{[n-\ell]}
\cb_j^{[n-\ell]}+\sum_{j\in \ZZ}\Bb_{k -2j}^{[n-\ell]} \db_j^{[n-\ell]},
\quad \ell=L,\dots,1.
\end{equation}
The following are equivalent ways to write the decomposition and  
reconstruction formulas, respectively:
{\small 
\begin{equation}\label{eq:dec_scheme_symbol}
{ \left\{\begin{array}{l}
\displaystyle{\cb^{[n-\ell]}(z^2) =\frac 12\left(
(\tilde\Ab^{[n-\ell]})^\sharp(z)\cb^{[n-\ell+1]}(z)+
(\tilde\Ab^{[n-\ell]})^\sharp(-z)\cb^{[n-\ell+1]}(-z)\right)},\\\noalign{\medskip}
\displaystyle{\db^{[n-\ell]}(z^2) =\frac 12\left(
	(\tilde\Bb^{[n-\ell]})^\sharp(z)\cb^{[n-\ell+1]}(z)+
	(\tilde\Bb^{[n-\ell]})^\sharp(-z)\cb^{[n-\ell+1]}(-z)\right)},
\end{array}\right.}
\end{equation}
}
{and}
\begin{equation*}
\cb^{[n-\ell+1]}(z)= \Ab^{[n-\ell]}(z)\,
\cb^{[n-\ell]}(z^2)+\Bb^{[n-\ell]}(z)\,
\db^{[n-\ell]}(z^2).
\end{equation*}

As mentioned, there is a close connection with  vector \emph{subdivision schemes}. In particular, in (\ref{eq:rec_scheme}), the action of the low-pass reconstruction filter $\Ab^{[n]}$ at  each level is nothing else than the action of a \emph{vector subdivision operator} $S_{\Ab^{[n]}}$. 
This allows for efficient constructions   of wavelet systems. In fact, given a subdivision operator satisfying some mild assumptions, the associated mask can be completed to a biorthogonal system. This completion, as we will see, is particularly straightforward if the scheme is interpolatory.


\section{Hermite subdivision  preserving exponentials and polynomials}
\label{sec:hermite}
Since our aim  is to propose an MRA based on  Hermite
subdivision schemes, we recall some basic facts on such schemes, focusing on subdivision preserving exponential and polynomial data.

Let $\Db$ be the diagonal matrix
\begin{equation*}
\Db = 
\left[
\begin{array}{ccccc}
1& 0 & 0 &\cdots & 0\\
0& \frac12 &0&\cdots & 0\\ 
\vdots \\ 
0& 0&   0& \cdots& \frac{1}{2^d}
\end{array}
\right].
\end{equation*}

An Hermite subdivision scheme $S(\Ab^{[n]}:\, n\ge 0)$ consists of the successive applications of level-dependent subdivision operators, which produce, starting from an initial sequence $\cb^{[0]}$,  sequences of sequences as
\begin{equation}  \label{eq:HermSubd0}
\nDb{n+1} \ncb{n+1}_j
=
\sum_{k\in \ZZ} \nAb{n}_{j-2k} \nDb{n} \ncb{n}_{k}=:
\cS_{\nAb{n}} \nDb{n} \ncb{n},\quad n\in \NN,
\end{equation}
with the special assumption  that, at each level, the sequence $\cb^{[n]}$ is related
to the evaluations of some function and its derivatives up to order
$d$ on the grid $2^{-n}\ZZ$.

An Hermite scheme is said to be {\it interpolatory} if 
 $(S_{\Ab^{[n]}} \cb^{[n]})_{2j}=\cb^{[n]}_{j}$, or, on the symbol side,
\begin{equation}\label{eq:interpolatoryscheme}
\nAb{n}(z) + \nAb{n}(-z) = 2\Db.
\end{equation}
In such a situation, all the even-indexed mask coefficients are zero
matrices, except $\nAb{n}_0$, which is $\Db$.

The scheme (\ref{eq:HermSubd0}) is said to be $C^d$-convergent
if for any vector-valued sequence $\fb_0\in \ell^d_\infty(\ZZ)$ and the
corresponding sequence of refinements $\fb_{n+1}=\cS_{\nAb{n}}\fb_n$,
there exists a uniformly continuous vector field
$\bphi:\RR\rightarrow\RR^{d+1}$, such that
\[
 \lim_{n\rightarrow\infty}
 \sup_{\alpha\in\ZZ}|\bphi(2^{-n}\alpha)-\fb_n(\alpha)|_{\infty}
 = 0
\]
with
$\phi_0\in C^d_u(\RR)$ and $\frac{d^j\phi_0}{dx^j}=\phi_j$
for $j=0,\dots,d$.

In case of $C^d$-convergence, the special  choice
of delta sequences
as initial data  produces, in the limit, the so-called
\emph{basic limit function} of the Hermite subdivision scheme, that is
the $(d+1)\times (d+1)$ matrix-valued function $\Fb$ given by
$$
\Fb=\left[
  \begin{array}{cccc}
    \phi_0 & \phi_1 & \dots & \phi_d\\
    \phi_0^\prime & \phi_1^\prime & \dots & \phi_d^\prime\\
    \vdots & & &\\
    \phi_0^{(d)} & \phi_1^{(d)} & \dots & \phi_d^{(d)}
  \end{array}
\right]
$$
with $\phi_j \in  C_u^d (\RR)$, $j=0,\ldots,d$.

In this case, all the schemes
$S(\Ab^{[n]}:\, n\ge \ell)$ for $\ell\ge 0$ are $C^d$-convergent, each
with basic limit function $\Fb^{[\ell]} $, where $\Fb^{[0]} $ coincides
with $\Fb$. Furthermore, similar arguments as in \cite{DubucMerrien06,DynLevin}, show that the functions $\Fb^{[\ell]}$ are related by
the refinement equations
\begin{equation}\label{eq:refinablebasic}
\Fb^{[n-1]}=\sum_k \Db^{-1}\Fb^{[n]}(2\cdot-k)\Ab_k^{[n-1]}.
\end{equation}

The refinement property (\ref{eq:refinablebasic}) is closely connected to the
possibility of considering a level-dependent (nonstationary)  Hermite multiresolution
analysis, where each space $V_n$ is spanned by the translates of the
functions $\phi_0^{[n]}(2^n\cdot),\dots, \phi_d^{[n]}(2^n\cdot)$.

 Recently, in \cite{ContiCotroneiSauer15}, Hermite schemes
preserving elements of the space
\begin{equation*}
V_{p,\Lambda}  = \mbox{\rm span}
\left\{1,x,\ldots,x^{p},e^{\pm\lambda_1 x},\cdots, e^{\pm\lambda_r x}\right\}
\end{equation*}
for $\Lambda:=\{\lambda_1,\dots,\lambda_r\}$ with
$\lambda_j\in \CC$, $j=1,\dots,r$, and $d=p+2r$
have been studied.

In particular, this preservation property has been related to the
factorization of the subdivision operator or, equivalently, of the corresponding symbol in terms of the so-called \emph{annihilator} or \emph{cancellation operator}. Such factorization
 turns out to be useful in deriving preservation and
cancellation properties for the MRA based decomposition and reconstruction
schemes presented in the previous section. 

The polynomial and exponential preservation property is expressed in
terms of the so-called \emph{$V_{p,\Lambda}$-spectral condition}, as
in \cite{ContiCotroneiSauer15}, in the sense that the subdivision
operator $\cS_{\Ab^{[n]}}$ satisfies:
\begin{equation}\label{eq:spectral}
\cS_{{\Ab}^{[n]}} \vb^{[n]}_{f;k} = \vb_{f;k}^{[n+1]},\qquad  f \in
V_{p,\Lambda}, \, n\ge 0,
\end{equation}
where for $f \in C^d (\RR)$ we denote by $\vb_f^{[n]}$ the vector
sequence with
$$	
\vb^{[n]}_{f;k} :=
\left[
\begin{array}{c}
f( 2^{-n} k ) \\ 2^{-n} f'( 2^{-n} k )\\ \vdots \\
2^{-nd} f^{(d)}(2^{-n} k )
\end{array}
\right], \qquad k \in \ZZ.
$$
In terms of symbols, (\ref{eq:spectral}) reads as:
\begin{equation}
\label{eq:spectral_symbol}
\Ab^{[n]}(z) \vb^{[n]}_{f}(z^2) = \vb^{[n+1]}_{f}(z)\end{equation}
with  $\vb^{[n]}_{f}(z) =\sum_k \vb^{[n]}_{f;k} z^k$.

For standard Hermite schemes, i.e., schemes preserving only polynomials,  it has been shown in \cite{MerrienSauer12}
that the  preservation property is related to the factorization
of the symbol in terms of the so-called  \emph{complete Taylor operator} $\cT_p$, whose symbol is given by
\begin{equation*}
\Tb_p(z):=\left[ 
\begin{array}{*5{C{3.9em}}}   
(z^{-1}-1) &-1 & \cdots &-\frac 1{(p-1)!} &-\frac 1{p!}\\
0 & (z^{-1} - 1) &\ddots & \vdots & \vdots\\
\vdots &   &\ddots & -1& \vdots\\[0.5em]
0& \dots  &  & (z^{-1} - 1)& -1\\[0.5em]
0   & \dots & & 0 & (z^{-1}-1)
\end{array}
\right].
\end{equation*}

In \cite{ContiCotroneiSauer15}, a similar result for Hermite schemes
preserving both polynomial and exponential data is given, in terms of the following convolution operator.

\begin{definition}\label{def:cancellation}
The \emph{level-$n$ cancellation  operator}
$\cH^{[n]}_{p,\Lambda} : \ell^{d+1}(\ZZ) \to \ell^{d+1}(\ZZ)$
is defined as a convolution operator satisfying
\begin{equation*}
\left(\cH^{[n]}_{p,\Lambda} \vb^{[n]}_f\right)_j
=
\sum_{k\in \ZZ}\Hb^{[n]}_{p,\Lambda;j-k}\,\vb^{[n]}_{f;k}
=0,
\qquad f\in V_{p,\Lambda},
\end{equation*}
or, equivalently, in terms of symbols
\begin{equation}\label{eq:AnnOp}
\Hb^{[n]}_{p,\Lambda}(z)\,\vb^{[n]}_{f}(z)
=0,
\qquad f\in V_{p,\Lambda}.
\end{equation}
\end{definition}

More specifically, the following theorem has been proved.

\begin{theorem}
If the subdivision operator $\cS_{\Ab^{[n]}}$  satisfies the
$V_{p,\Lambda}$-spectral condition, then there exists a finitely
supported  mask $\Rb^{[n]}\in \ell^{(d+1)\times (d+1)}(\ZZ)$
such that
\begin{equation*}
\cH^{[n+1]}_{p,  \Lambda} \cS_{\Ab^{[n]}} = \cS_{\Rb^{[n]}}
\cH^{[n]}_{p, \Lambda}
\end{equation*}
or, in terms of symbols, 
\begin{equation}\label{eq:factSA}
\Hb_{p,\Lambda}^{[n+1]}(z)\Ab^{[n]}(z)
=
\Rb^{[n]}(z)\Hb^{[n]}_{p,\Lambda}(z^2).
\end{equation}
\end{theorem}
As shown in \cite{ContiCotroneiSauer15}, 
the level-$n$ cancellation operator can be obtained as
$$
\cH^{[n]}_{p,\Lambda}=\cH_{p,2^{-n}\Lambda},
$$
where $\cH_{p,\Lambda}$ is the unique minimal operator whose symbol
$\Hb_{p,\Lambda} ^*(z) \in \RR^{(d+1)\times(d+1)}$ has the following
structure
\begin{equation}\label{eq:Hlambda1}
\Hb_{p,\Lambda}^*(z)
= \left[ \begin{array}{cc}
\Tb_{p}(z) & \ast \\ 0 & \ast 
\end{array}\right]
\end{equation}
and satisfies
\begin{equation}\label{eq:Hlambda2}
\Hb_{p,\Lambda}^* \left( e^{\mp \lambda} \right)
\left[\begin{array}{c}
1 \\ \pm \lambda \\ \vdots \\ (\pm \lambda)^d
\end{array}\right]
= 0.
\end{equation}
The remaining blocks in (\ref{eq:Hlambda1}) can be explicitly
computed (see \cite{ContiCotroneiSauer15} for details).
As an example, we give the expressions of $ \Hb^*_{p,\Lambda} $
in the cases $p=0$ and $p=1$, considering only one pair of
frequencies $\lambda,-\lambda$:
\begin{equation}\label{esempio:H_2}
\Hb_{0,\{\lambda\}} ^*(z)
=
\left[ \begin {array}{ccc} 
{z}^{-1}-1
&
-\displaystyle{\frac {\sinh \left( \lambda \right) }{\lambda}}
&
\displaystyle{\frac {1-\cosh \left( \lambda \right)}{{\lambda}^{2}}}
\\
\noalign{\medskip}
0
&
{z}^{-1}-\cosh \left( \lambda\right)
&
-\displaystyle{\frac {\sinh \left( \lambda \right) }{\lambda}}
\\
\noalign{\medskip}
0
&
-\lambda\,\sinh \left( \lambda \right) 
&
{z}^{-1}-\cosh \left( \lambda \right) 
\end {array} \right], 
\end{equation}
\begin{equation}\label{esempio:H_3}
\Hb_{1,\{\lambda\}} ^*(z)
=
\left[\begin{array}{cccc}
{z}^{-1}-1
&
-1
&
\displaystyle{\frac{1-\cosh \left(\lambda \right)}{{\lambda}^{2}}}
&
\displaystyle{\frac{\lambda-\sinh\left(\lambda\right)}{{\lambda}^{3}}}
\\
\noalign{\medskip}
0 & & & \\
\noalign{\medskip}
0 & &  \Hb_{0,\{\lambda\}} ^*(z) &   \\
\noalign{\medskip}
0 & & &
\end{array}\right].
\end{equation}
In \cite{ContiCotroneiSauer15}, it has also been proved that the operators
$\cH_{p,2^{-n}\Lambda}$ reduce to Taylor operators as the frequencies tend to zero; as a consequence, the asymptotical behavior of such operators is easily found as
\begin{equation*}
\label{eq:HdLconvTd3}
\lim_{n\to \infty} \cH_{p,2^{-n}\Lambda} = \cT_d,\quad d=p+2\#\Lambda.
\end{equation*}

\section{MRA based on Hermite subdivision}
\label{sec:MRA_fact}
In this section we describe how to build  multiresolution analyses associated to a convergent Hermite  level-dependent scheme, where the subdivision operator plays the role of the reconstruction low-pass filter.
If such operator has the polynomial-exponential preservation property, then the wavelet decomposition filter can be easily constructed in order to cancel elements in the space $V_{p,\Lambda}$.

Before going into the details of the discussion, let us give the  vanishing moment definition for the wavelet filter with respect to the elements in the space $V_{p,\Lambda}$.

\begin{definition}
A level-dependent multiwavelet analysis filter  $\widetilde \Bb^{[n]}$ satisfies the \emph{$V_{p,\Lambda}$-vanishing moment condition} if
\begin{equation} \label{eq:vanmomBT}
	\sum_{j\in \ZZ}(\widetilde\Bb_{j -2k}^{[n]})^T
\vb^{[n+1]}_{f;j}=0, \qquad  f \in
V_{p,\Lambda}, \, n\ge 0,
\end{equation}
\end{definition}

A nice property of such filters is that the symbol can be factorized in a straightforward way. In fact, since 
$ {\cal H}^{[n+1]}_{p,\Lambda}$, as defined in Section \ref{sec:hermite}, is the  minimal (convolution)  annihilator for the elements in $V_{p,\Lambda}$ at the level $n+1$, the following result follows.

\begin{proposition}
The filter  $\widetilde \Bb^{[n]}$ satisfies the $V_{p,\Lambda}$-vanishing moment condition if and only if there exists a finite filter $\Sb^{[n]}\in \ell^{(d+1)\times (d+1)}(\ZZ)$ such that
 \begin{equation}\label{eq:factorBT}
 (\widetilde \Bb^{[n]})^ {\sharp} (z)= \Sb^{[n]}(z) \Hb^{[n+1]}(z).
 \end{equation}	
\end{proposition}

One possibility of constructing a biorthogonal multiwavelet analysis filter within an Hermite-type framework, with the property (\ref{eq:vanmomBT}), is by taking the Hermite subdivision operator as reconstruction filter. Since its  symbol satisfies the factorization $\Ab^{[n]}(z)=(\Hb^{[n+1]})^{-1}(z)\Rb(z) \Hb^{[n]}(z^2)$,  if we    impose the factorization  (\ref{eq:factorBT}), then it follows that the third of the   biorthogonality conditions (\ref{eq:biortcond}) is satisfied if and only if $\Sb(z)$ is chosen such that
$$\Sb^{[n]}(z)\Rb^{[n]}(z)+\Sb^{[n]}(-z)\Rb^{[n]}(-z)=0.$$

A good alternative to this kind of procedure, is offered by Hermite interpolatory schemes, which allow an easier construction of an Hermite-type level-dependent MRA.
Let ${\cal S}_{\Ab^{[n]}}$ be the $n$-th level subdivision operator
associated to a $C^d$-convergent interpolatory Hermite subdivision
scheme. 
As stated in Section~\ref{sec:hermite}, in this case, there exists
a sequence of basic matrix limit functions $(\Fb^{[n]}:\: n\ge 0)$,
whose first rows correspond to  vector-valued functions
$
\bPhi^{[n]}=
\big[ \phi_0^{[n]},\phi_1^{[n]},\ldots, \phi_d^{[n]}\big]^T
$
satisfying the refinement relations \eqref{eq:primal_2_scale}
and the Hermite interpolatory conditions
$$\big(\bPhi^{[n]}\big)^{(j)}(k)=\eb_j\delta_{0k}$$
with $\eb_j$ denoting the $j$-th coordinate vector.
As in \eqref{eq:primal_MRA}, they span a level-dependent MRA
$(V_n:n\in\NN)$ for the space $C^d_u$ of  uniformly $C^d$-continuous
functions.

The projection of a generic $f\in C^d_u(\RR)$ on  $V_n$ is defined
in terms of the Hermite interpolant
\begin{equation}\label{eq:projV}
{\cal P}_n f
=
\sum_{k\in \ZZ} \big(\bPhi^{[n]}\big)^T(2^n\cdot -k) \Db^n
\left[ \begin{array}{c}
f(2^{-n}k)\\f'(2^{-n}k)\\\vdots \\ f^{(d)}(2^{-n}k)
\end{array}\right].
\end{equation}
The (multi)wavelet spaces $W_n$ can be defined as the
complementary spaces of $V_n$ in $V_{n+1}$, and, from the
decomposition formula
$$
{\cal P}_{n+1}f
=
{\cal P}_{n}f+( {\cal P}_{n+1}- {\cal P}_{n})f,
$$
we can define the action of projection operator on $W_n$ as
$$
{\cal Q}_n f
=
({\cal P}_n-{\cal P}_{n-1}) f.
$$

It is now easy to find the filters involved in the discrete wavelet
decomposition \eqref{eq:dec_scheme} and the reconstruction
\eqref{eq:rec_scheme} scheme associated to such a MRA.
Let $f\in V_{n+1}=V_{n}\oplus W_{n}$ be given in terms of a
(vector-valued) coefficient sequence  $(\cb^{[n]}_k:k\in \ZZ)$,
that is
$$f=\sum \bphi(2^{n+1}\cdot-k)\cb^{[n+1]}_k.$$
In order to find the coefficient sequence $(\cb^{[n]}_k:k\in \ZZ)$
representing $f$ in $V_{n}$, we just compare the actions of the
projection operators \eqref{eq:projV} on $V_{n}$ and $V_{n+1}$.
We get
$$
\cb^{[n]}_k
=\Db^{n} \left[\begin{array}{c}
f(2^{-n}k)\\f'(2^{-n}k)\\\vdots \\ f^{(d)}(2^{-n}k)
\end{array}\right]
=\Db^{-1}\Db^{n+1} \left[ \begin{array}{c}
f(2^{-n-1}(2k))\\f'(2^{-n-1}(2k))\\\vdots \\ f^{(d)}(2^{-n-1}(2k))
\end{array}\right]
=\Db^{-1} \cb^{[n+1]}_{2k}.
$$
Thus, the low-pass decomposition step consists of just subsampling
the rescaled sequence $(D^{-1}\cb^{[n+1]}_k:k\in \ZZ)$
by a factor of 2.
Using the symbol formalism, this is equivalent to the identity
\begin{equation}
\label{eq:clink}
\cb^{[n]}(z^2)=\frac 12 \Db^{-1}\left(\cb^{[n+1]}(z) +\cb^{[n+1]}(-z) \right).
\end{equation}

In order to find the high-pass wavelet coefficients
$(\db^{[n]}_k:k\in \ZZ)$ involved in the representation of  $f$
in $ W_{n}$, we first observe that, in view of the refinability of
the functions $\bPhi^{[n]}$ and the above formula for the
decomposition step, we have
\begin{eqnarray*}
{\cal P}_{n}f
&=&
\sum_{\ell}\big(\bPhi^{[n]}\big)^T(2^{n}\cdot -\ell)\cb_\ell^{[n]}
\\
&=&
\sum_k \big(\bPhi^{[n+1]}\big)^T(2^{n+1}\cdot -k)
	\sum_{\ell} \Ab^{[n]} _{k-2\ell}\cb_\ell^{[n]}
\\
&=&
\sum_k \big(\bPhi^{[n+1]}\big)^T(2^{n+1}\cdot -k)
	\sum_{\ell} \Ab^{[n]} _{k-2\ell}\Db^{-1}\cb_{2\ell}^{[n+1]}.
\end{eqnarray*} 
Thus, we have
\begin{eqnarray*}
{\cal Q}_{n}f
&=&
({\cal P}_{n+1}-{\cal P}_{n})f
\\
&=&
\sum _k \big(\bPhi^{[n+1]}\big)^T(2^{n+1}\cdot -k)
 	\Big( \cb_k^{[n+1]}- \sum_{\ell} \Ab^{[n]} _{k-2\ell}\Db^{-1}
 			\cb_{2\ell}^{[n+1]}\Big)
\\
&=&
\sum _k \big(\bPhi^{[n+1]}\big)^T(2^{{n+1}}\cdot -k)\db_{k}^{[n]},
\end{eqnarray*}
which produces the following formulas
$$ \db_{2k}^{[n]}=0,$$
\begin{eqnarray*}
\db_{2k+1}^{[n]}
&=&
\cb_{2k+1}^{[n+1]} - \sum_{\ell} \Ab^{[n]} _{2k+1-2\ell}\Db^{-1}
\cb_{2\ell}^{[n+1]}
\\
&=&
\cb_{2k+1}^{[n+1]}-\Big( \sum_{\ell} \Ab^{[n]} _{2k+1-\ell}\Db^{-1}
\cb_{\ell}^{[n]}-\Ab^{[n]}_0 \Db^{-1}\cb_{2k+1}^{[n+1]}\Big)
\\
&=&
2\cb_{2k+1}^{[n+1]}-
\sum_{\ell} \Ab^{[n]} _{2k+1-\ell}\Db^{-1}\cb_{\ell}^{[n+1]}.
\end{eqnarray*} 
So the wavelet coefficients are obtained by means of a convolution
with the filter $2\Ib-\Ab^{[n]} \Db^{-1}$ followed by a shift and
subsampling.

As to the reconstruction part, the coefficients in the finer space
$V_{n+1}$ are easily obtained as a sum of the upsampled (shifted)
wavelet coefficients $\db^{[n]}$ and the coefficients generated
by the subdivision operator ${S_{\Ab^{[n]}}}$ applied  to
$\cb^{[n]}$.

Since in case of interpolatory Hermite subdivision schemes
$\Ab_0^{[n]}=\Db$, the pair of decomposition filters, in terms
of symbols, is given by
\begin{equation}\label{eq:decfilters}
\begin{array}{l}
\tilde \Ab^{[n]}(z)=\Db^{-1},
\\
\tilde \Bb^{[n]}(z)= z (2\Ib -\Db^{-1} (\Ab^{[n]})^\sharp(z))=
    z \Db^{-1} (\Ab^{[n]})^\sharp(-z),
    \end{array}
\end{equation}
while the reconstruction filters are
$$\Ab^{[n]}(z), \quad \Bb^{[n]}(z)=z\Ib.$$
One can easily check that they satisfy the biorthogonality  conditions.
%

The previous arguments allow us to  show that, starting from a Hermite subdivision operator
${S_{\Ab^{[n]}}}$ preserving polynomial/exponential data, one can always find a complete wavelet system where the  high-pass filter involved in the
decomposition has the property of
canceling those polynomials and exponentials.

\begin{proposition}\label{thr:fac_Bt}
Let $\Ab^{[n]}$ be the mask of an Hermite
subdivision scheme satisfying the $V_{p,\Lambda}$-spectral
condition. Using $\Ab^{[n]}$ as low-pass synthesis filter, there
exists a biorthogonal filter bank such that the symbol
of the high-pass analysis filter $\tBb^{[n]}$ satisfies the  $V_{p,\Lambda}$-vanishing moment condition.
\end{proposition}

\begin{proof}
	The existence of the filter bank is already proven by the
	construction via the ``prediction-correction'' approach by taking $\Ab^{[n]}$ as the mask of an interpolatory Hermite
	subdivision.
	To check  that the analysis wavelet filter annihilates the elements in the space $V_{p,\Lambda}$, we fix $\cb^{[n]}=\vb^{[n]}_{f}$ with $f\in V_{p,\Lambda}$, $n\ge 0$. We  observe that, from (\ref{eq:dec_scheme_symbol}),
	(\ref{eq:decfilters}), (\ref{eq:interpolatoryscheme}) and
	(\ref{eq:clink}),
\begin{eqnarray*}	
	\db^{[n]}(z^2) &=&\frac 1{2z}\left(
		\Ab^{[n]}(-z)\Db^{-1}\cb^{[n+1]}(z)-
		\Ab^{[n]}(z)\Db^{-1}\cb^{[n+1]}(-z)\right)\\
			&=&	\frac 1{2z}\left(
			2 \cb^{[n+1]}(z)-
			\Ab^{[n]}(z)\Db^{-1}(\cb^{[n+1]}(z)+\cb^{[n+1]}(-z))\right)\\
		&=&	\frac 1{z}\left(
			\cb^{[n+1]}(z)  -\Ab^{[n]}(z)
		\cb^{[n]}(z^2)\right),
	\end{eqnarray*}
which is identically zero because of the preservation property (\ref{eq:spectral_symbol}) of $\Ab^{[n]}$.
\end{proof}

It follows that the filter associated to $(\Bb^{[n]}(z))^{\sharp}$ is a cancellation operator for the elements in $V_{p,\Lambda}$ at the level $n+1$. Since
$ {\cal H}^{[n+1]}_{p,\Lambda}$, as defined in Section \ref{sec:hermite}, is the  minimal (convolution)  annihilator for such space, there exists a matrix polynomial $\Sb^{[n]}(z)$ such that
 \begin{equation*}
 (\widetilde \Bb^{[n]})^ {\sharp} (z)= \Sb^{[n]}(z) \Hb^{[n+1]}(z).
 \end{equation*}
The structure of the polynomial $\Sb^{[n]}(z)$  is  easily found. We get
\begin{eqnarray*}
\Sb^{[n]}(z) = -z^{-1}\big(\Hb^{[n+1]}(-z)\big)^{-1}\Rb^{[n]}(-z)\Db^{-1}\Hb^{[n+1]}(-z),
\end{eqnarray*}
where we use \eqref{eq:decfilters}, \eqref{eq:factSA}, and the following
proposition.


\begin{proposition}\label{thr:prop_H}
The symbol of the level-n cancellation operator $\cH^{[n]}_{p,\Lambda}$
satisfies
\[
\Hb^{[n]}(z^2)\Db^{-1}
=
-\Db^{-1}\Hb^{[n+1]}(-z)\Hb^{[n+1]}(z).
\]
\end{proposition}
\begin{proof}
Let $\cb^{[n]}(z):=\vb^{[n]}_f(z)$ for $f\in V_{p,\Lambda}$. 
We want to connect $\Hb^{[n+1]}(-z)\Hb^{[n+1]}(z)$ and $\Hb^{[n]}(z^2)$
through the identity \eqref{eq:clink},
that is
\begin{equation*}
\cb^{[n]}(z^2)=\frac 12 \Db^{-1}\left(\cb^{[n+1]}(z) +\cb^{[n+1]}(-z) \right).
\end{equation*}
Using \eqref{eq:AnnOp}, we have
\begin{eqnarray}\label{eq:Ann_n+1}
\Hb^{[n+1]}(-z)\Hb^{[n+1]}(z)\cb^{[n+1]}(z)
& =  & 0,\\\label{eq:Ann_n+1-}
\Hb^{[n+1]}(z)\Hb^{[n+1]}(-z)\cb^{[n+1]}(-z)
& =  & 0.
\end{eqnarray}
Since $\Hb^{[n]}(z)=z^{-1}\Ib+\Hb_0^{[n]}$, we get
\begin{equation}\label{eq:rep_H_n+1^2}
\Hb^{[n+1]}(-z)\Hb^{[n+1]}(z)
 = 
z^{-2}\Ib-(\nHb{n+1}_0)^2
\end{equation}
and, therefore,
$\Hb^{[n+1]}(-z)\Hb^{[n+1]}(z)=\Hb^{[n+1]}(z)\Hb^{[n+1]}(-z)$.
This gives us together with \eqref{eq:Ann_n+1-} and \eqref{eq:Ann_n+1} that
\begin{equation*}
\Hb^{[n+1]}(-z)\Hb^{[n+1]}(z)\big(\cb^{[n+1]}(z) +\cb^{[n+1]}(-z)\big)
= 0.
\end{equation*}
Using the identity \eqref{eq:clink}, we get
\begin{equation*}
\Hb^{[n+1]}(-z)\Hb^{[n+1]}(z)\Db\cb^{[n]}(z^2)= 0.
\end{equation*}
Thus, $\Hb^{[n+1]}(-z)\Hb^{[n+1]}(z)\Db$ is the symbol of a
cancellation operator for $\cb^{[n]}(z^2)$.
From \eqref{eq:AnnOp} we know that $\Hb^{[n]}(z^2)$ is the symbol of the minimal level-$n$ cancellation
operator, and, since $\supp \Hb^{[n+1]}(-z)\Hb^{[n+1]}(z) =[-2,0] =
\supp \Hb^{[n]}(z^2)$, there exists a constant matrix $\Kb$ such that
\begin{equation*}
\Kb\Hb^{[n]}(z^2)
=
\Hb^{[n+1]}(-z)\Hb^{[n+1]}(z)\Db.
\end{equation*}
Comparison of \eqref{eq:rep_H_n+1^2} with $\Hb^{[n]}(z^2)=z^{-2}\Ib+\Hb_0^{[n]}$ gives us that $\Kb=-\Db$, which proves the proposition.
\end{proof}

\section{A family of interpolatory Hermite wavelets}\label{sec:HermiteInterp}
We now derive a biorthogonal multiwavelet filter bank  based on the polynomial
and exponential reproducing Hermite subdivision scheme proposed  in \cite{ContiCotroneiSauer16}, whose construction is briefly recalled.

Such a scheme has been realized by proving the existence and uniqueness of the solution to  the Hermite interpolation problem 
 $$
 f^{(j)} (\epsilon) = y^j_{\epsilon}, \qquad j=0,\dots,d, \quad
 \epsilon \in \{ 0,1\},
 $$
 where $\yb_\varepsilon = \left( y^j_{\varepsilon} \,:\; j = 0,\dots,d
 \right)$ are given vectors of data,
in $V_{d+p+1,\Lambda}$ over the interval $[0,1]$.  
An explicit form of  the basis function vectors 
$\hb^{\epsilon}_{\Lambda}=[h^{\epsilon}_{0,\Lambda},\cdots, h^{\epsilon}_{d,\Lambda}]^T$, $\epsilon \in \{0,1\}$, has been given in \cite{ContiCotroneiSauer16}.

The local Hermite interpolant of a generic  vector-valued sequence of data $\fb_n$ 
$$
q^\alpha (\fb_n) (x) = \sum_{\epsilon \in \{0,1\}}
\fb_n^T (\alpha + \epsilon) \Db^n \hb^{\epsilon}_{2^{-n}\Lambda} ( 2^n x - \alpha )$$
over the interval $
\left[\frac {\alpha}{2^n}, \frac{\alpha+1}{2^n}\right]
$
is then evaluated at the midpoints, producing an Hermite subdivision  scheme whose 
$n$-th level mask $\Ab^{[n]}$ is given by
$$
\Ab^{[n]}(1)=\Db  \Nb_{0,2^{-n}\Lambda}\left(\frac 12 \right),
\quad  
\Ab^{[n]}(0)=\Db,
\quad
\Ab^{[n]}(-1)=\Db \Nb_{1,2^{-n}\Lambda}\left(\frac 12\right)
$$
with
$\Nb_{\epsilon,2^{-n}\Lambda}(y) := \left[ ( \hb_{\epsilon,2^{-n}\Lambda}^j )^{(k)}(y) \,:\, k,j =
  0,\dots,d \right]$, $ y\in [0,1].
  $

It is worthwhile to observe that, in the limit, such mask coincides with  the mask of the
Hermite  B-splines of degree $2d+1$.

Let us now restrict to the case $p=0$ and $r=1$.
The basis functions produce, for each $n$, a function vector 
supported on $[-1,1]$, which satisfies a level-dependent refinement equation as in (\ref{eq:primal_2_scale})  with coefficients 
given by the mask  $\Ab^{[n]}$. Such vector, at the level $n=0$, has components   explicitly given by:
%
%
%
%
%
%
%
%
	
{\scriptsize	
$$
\begin{array}{l}
\phi_0^{[0]}(x)=\left\{\begin{array}{l}
 \left( x+1 \right) ^{3} \left( 6\,{x}^{2}-3\,x+1 \right) , \, x\in [-1,0]\\ \noalign{\medskip}
 -6\,{x}^{5}+15\,{x}^{4}-10\,{x}^{3}+1 , \quad\quad x\in [0,1]
\end{array}\right.
\\ \noalign{\medskip}
\phi_1^{[0]}(x)=\left\{\begin{array}{ll}
- \left( x+1 \right) ^{3}x \left( 3\,x-1 \right), \quad\quad x\in [-1,0]\\ \noalign{\medskip}
{x}^{3} \left( 3\,{x}^{2}-7\,x+4 \right) C(\lambda)
-{\frac {{x}^{3} \left( {\lambda}^{2}{x}^{2}-2\,{\lambda}^{2}x+{
			\lambda}^{2}+12\,{x}^{2}-30\,x+20 \right) S ( \lambda
		) }{2\lambda}}+{\frac {S( \lambda\,x) }{\lambda
	}}
 , \quad\quad x\in [0,1]
\end{array}\right.
\\ \noalign{\medskip}
\phi_2^{[0]}(x)=\left\{\begin{array}{ll}
\frac12 {\left( x+1 \right) ^{3}{x}^{2}} , \quad\quad x\in [-1,0]\\  \noalign{\medskip}
-{\frac {{x}^{3} \left( {\lambda}^{2}{x}^{2}-2\,{\lambda}^{2}x+{
			\lambda}^{2}+12\,{x}^{2}-30\,x+20 \right) C( \lambda
		) }{{\lambda}^{2}}}+{\frac {{x}^{3} \left( 3\,{x}^{2}-7\,x+4 \right) S
		( \lambda) }{\lambda}}\\ \noalign{\medskip}
\hspace{5cm} +{\frac {C ( \lambda\,x) }{{
			2\lambda}^{2}}} +{\frac {6\,{x}^{5}-15\,{x}^{4}+10
		\,{x}^{3}-1}{{\lambda}^{2}}}
	, \quad\quad x\in [0,1]
	\end{array}\right.
\end{array}$$
}
where we have set  $C(t):=\cosh(t)$ and $S(t):=\sinh(t)$.
The substitution $\lambda \hookrightarrow \lambda 2^{-n}$ gives the function at a generic level $n$. In Fig. \ref{fig:ExpHermite}
the components of ${\bPhi}^{[0]}$ are shown, corresponding to two different values of $\lambda$.

\begin{figure}
	\begin{center}
		\includegraphics[width=2.5cm]{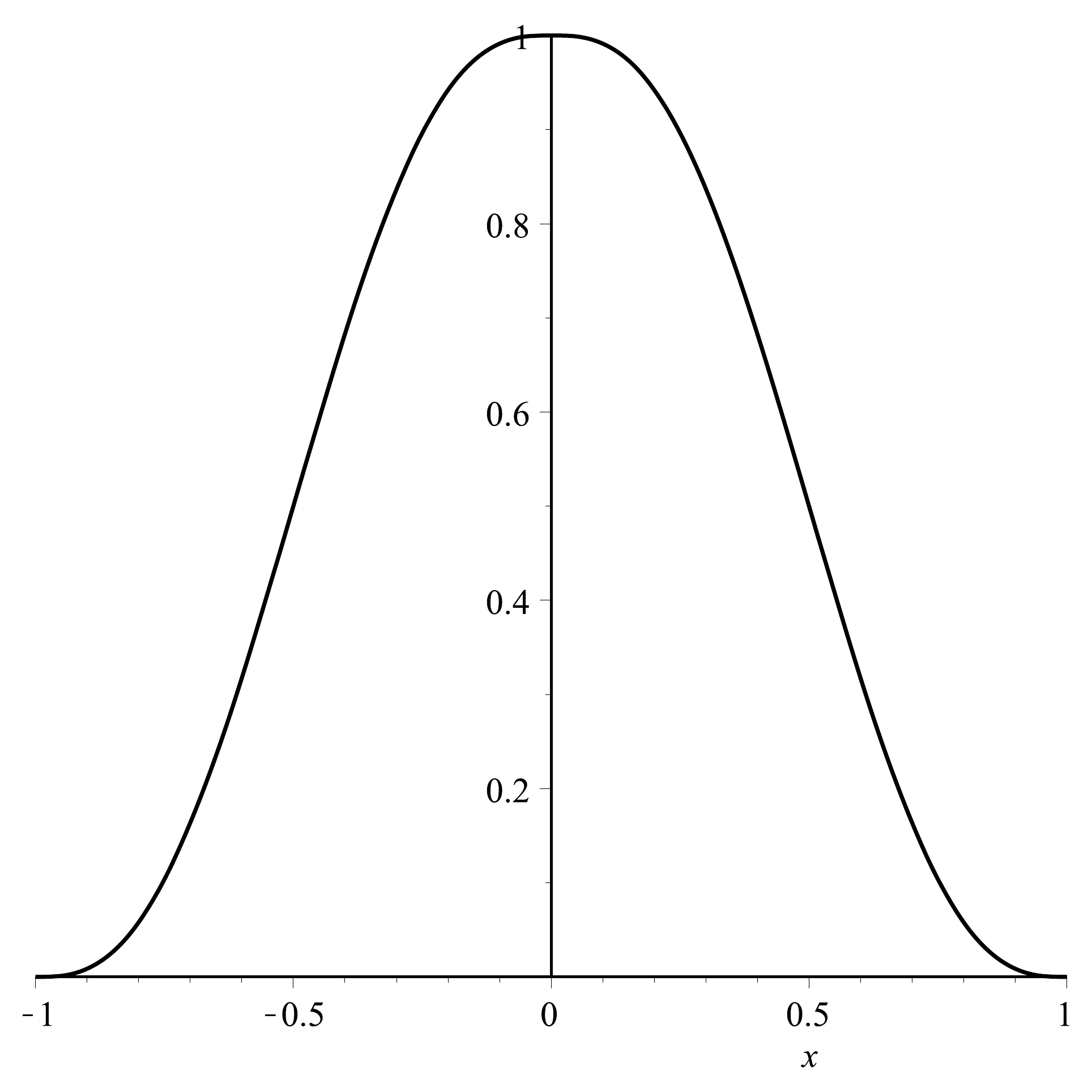}
			\includegraphics[width=2.5cm]{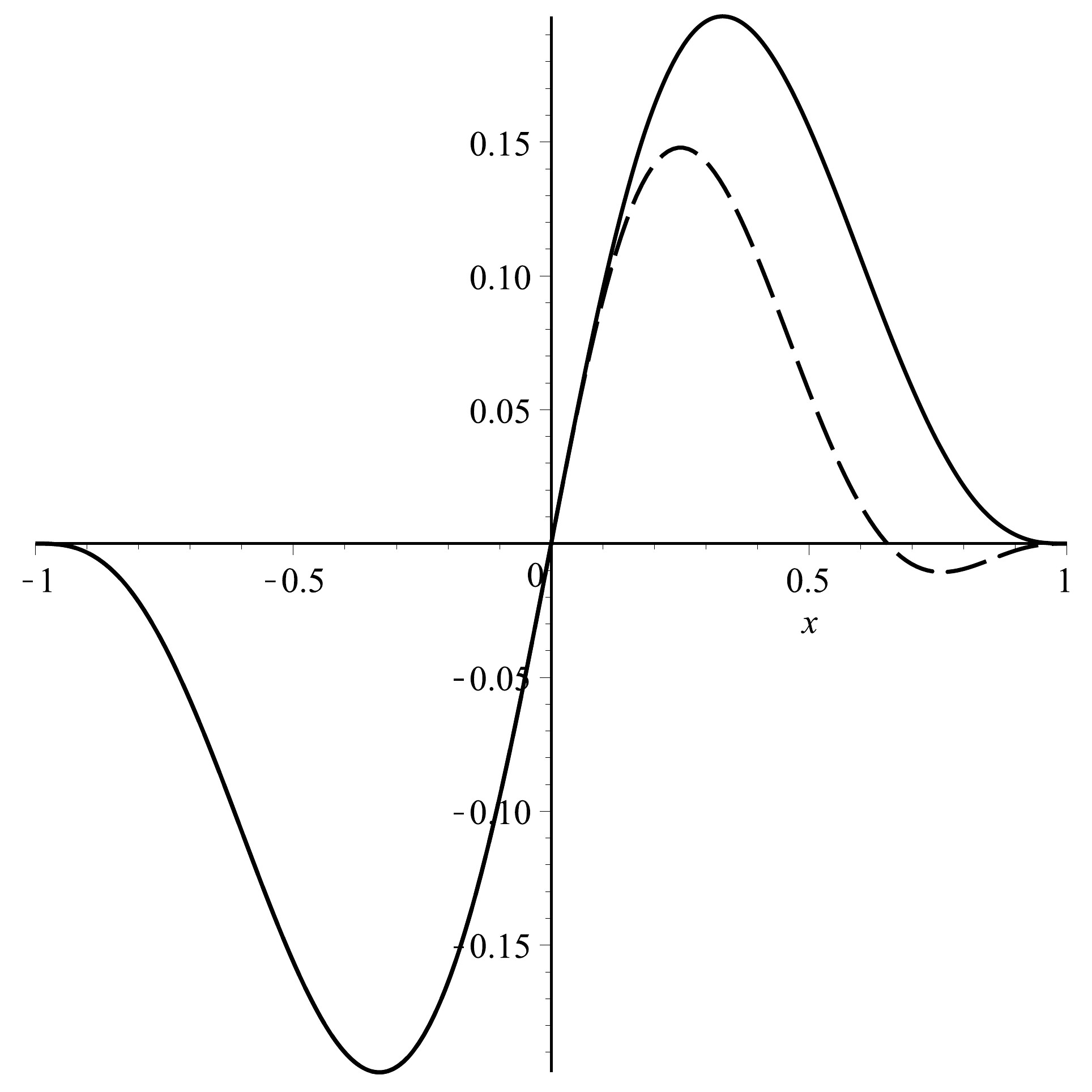}
				\includegraphics[width=2.5cm]{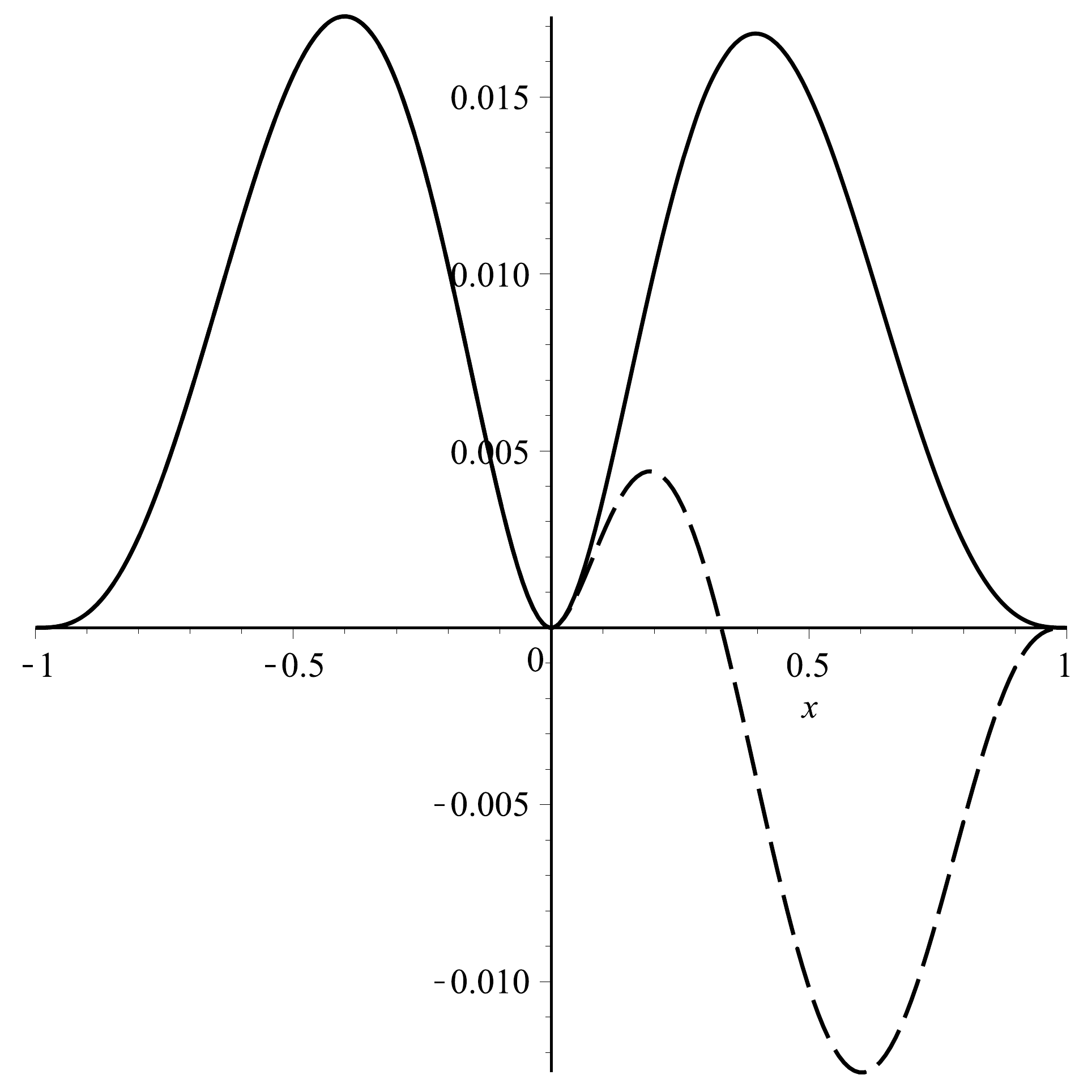}
		\caption{The three components of the exponential Hermite multi-scaling function, case $p=0$, $r=1$,  $n=0$, for   $\lambda=2$ 
			(solid line) and
		$\lambda=4$  (dashed line)}\label{fig:ExpHermite}
	\end{center}
\end{figure}

The corresponding low-pass decomposition mask has elements
$$
\Ab_{-1}:=\Ab^{[n]}_{-1} = 
\frac 1{64}
\left[ \begin {array}{ccc} 32&-10&1\\\noalign{\medskip}60&-14&1\\
\noalign{\medskip}0&24&-4\end {array} \right],\quad 
\Ab_0:=\Ab^{[n]}_{0} = 
\Db,
$$
$$
\Ab^{[n]}_1\!=\!
{\scriptsize \frac {1}{256}
\left[ \begin {array}{ccc} 
128 &
32\,\frac{2\,S(\frac{\lambda}{2})-S(\lambda)}{\lambda}+10\,C(\lambda)-\lambda\,S(\lambda)&
32\frac{2\,C(\frac{\lambda}{2})-C(\lambda)-1}{\lambda^2}+\frac{10\,S(\lambda)}{\lambda}-C(\lambda)
\\
\noalign{\medskip}
-240 &
\frac{32\,C(\frac{\lambda}{2})-60\,S(\lambda)}{\lambda}+14\,C(\lambda)-\lambda\,S(\lambda)&
60\,\frac{1-C(\lambda)}{\lambda^2}+\frac{32\,S(\frac{\lambda}{2})+14\,S(\lambda)}{\lambda}-C(\lambda)
\\
\noalign{\medskip}
0 &
12\lambda\,S(\frac{\lambda}{2})-18\,C(\lambda)+3\lambda\,S(\lambda) &
12\lambda\,C(\frac{\lambda}{2})-18\,S(\lambda)+3\lambda\,C(\lambda)   
\end {array} \right]},
$$
%
while the wavelet analysis filter taps are
$$\widetilde\Bb^{[n]}_k=(-1)^{(1-k)} \Db^{-1}(\Ab_{1-k}^{[n]})^T.$$

%
Their symbols admits the factorizations (\ref{eq:factSA}) and (\ref{eq:factorBT}), respectively, with respect to the cancellation operator (\ref{esempio:H_2}).

The limit functions
of the given example ($p=0$, $r=1$) result in the \emph{Hermite
quintic B-splines}, whose  connection with 
 multiwavelets  has already been widely studied for example in
\cite{Shumilov,Strela1995}. 

The corresponding filter has the symbol
$$\Ab(z)=\frac 1{64z}\left[ \begin {array}{ccc} 32\, \left( z+1 \right) ^{2}&10\, \left( z^2
-1 \right)  &{z}^{2}+1\\ \noalign{\medskip}-60\,
\left( z^2-1 \right)  &-2\left(\,{z}^{2}-16\,z+7\right)&-{z}^{2
}+1\\ \noalign{\medskip}0&-24\, \left( z^2-1 \right)  
 &-4\left(\,{z}^{2}-4\,z+1\right)
 \end{array} \right].
$$
Since, as already mentioned,  in the limit the annihilators $\cH_{p,2^{-n}\Lambda}$
reduce to the Taylor operator $\cT_d$,  we have
the factorization $\Tb(z)\Ab(z)=\Rb(z)\Tb(z^2)$, where
$$\Rb(z)=\frac 1{64}\left[ \begin {array}{ccc} 4(8-7\,z)&2(6\,z-5)&1\\ \noalign{\medskip}
60\,(1-z)&2\,(11z-7)&3\,z+1\\ \noalign{\medskip}0&24\,(1-z)&4\,(5z-1)
\end {array} \right] 
$$
and
$$\Tb(z)=\left[ \begin {array}{ccc} z^{-1}-1&-1&-\frac 12\\ \noalign{\medskip}
0&z^{-1}-1&-1\\ \noalign{\medskip}0&0&z^{-1}-1
\end {array} \right] .
$$
The symbol of the corresponding high-pass filter
of the decomposition, as constructed in Section~\ref{sec:MRA_fact}, is given by
$$
\tBb(z)=
\frac 1{16z^2}\left[ \begin {array}{ccc} -8\, \left( z-1 \right) ^{2}&-5\, \left( z^2
-1 \right) &-{z}^{2}-1\\ \noalign{\medskip}15\,
\left( z^2-1 \right)   &7\,{z}^{2}+16\,z+7&{z}^{2}-1
\\ \noalign{\medskip}0&12\, \left( z^2-1 \right)   &4(
\,{z}^{2}+4\,z+1)\end {array} \right] 
$$
and it is verified that it satisfies the factorization $\tBb(z)=\Sb(z)\Tb(z)$ with
$$
\Sb(z)=\frac 1{32z}\left[ \begin {array}{ccc} 16\,(z-1)&2(5-3z)&-2\\ \noalign{\medskip}
-30\,(z+1)&2\,(8z+7)&-3\,z-2\\ \noalign{\medskip}0&-24\,(z+1)&8\,(2z+1)
\end {array} \right] .
$$
We remark that such filters have been obtained with a  completely different approach than others in literature \cite{Shumilov,Strela1995}. Furthermore, the related factorization issues have never been studied before.

\section*{Conclusion}\label{sec:conclusion}
In this paper we have presented a Hermite-type multiwavelet system satisfying the
vanishing moment property with respect to elements in the space spanned by exponentials and polynomials. These systems  naturally generate MRAs which differ from the classical ones, in the sense that they are of nonstationary type, and the decomposition-reconstruction rules change accordingly to the level. In addition, for such kind of Hermite multiwavelets some nice results connected to the factorization of the corresponding filter symbol can be derived, exploiting their connection with Hermite subdivision. An example of such an Hermite multiwavelet system has been explicitly described. It includes,  as a particular case, 
the well-known finite element multiwavelets, which possess only polynomial vanishing moment properties and whose factorization issues have never been studied before.
Future researches include the application  of such  filter bank systems to some specific signal processing problems, where Hermite data are available (for example in problems of motion control) and where such data exhibit not just polynomial but also transcendental features.


\section*{References}

\begin{thebibliography}{10}

\bibitem{BCL}
S. Bacchelli, M. Cotronei, D. Lazzaro, An algebraic construction of k-balanced multiwavelets via the lifting scheme, Num. Alg. 23, (2000), 329--356.

\bibitem{BCS} S. Bacchelli, M. Cotronei, T. Sauer, Multifilters with and without prefilters, {BIT} 42:2 (2002), 231-261.

\bibitem{CohenDyn}
A. Cohen, N. Dyn,
Nonstationary subdivision schemes and multiresolution analysis,
SIAM J. Math. Anal. 26 (1996), 1745-1769.

\bibitem{CCS2008} C. Conti, M. Cotronei, T. Sauer, Full rank positive matrix symbols: interpolation and orthogonality,
{BIT} 48 (2008), 5--27.

\bibitem{CCS2010} C. Conti, M. Cotronei, T. Sauer, Full rank interpolatory subdivision schemes: Kronecker, filters
and multiresolution, {J. Comput. Appl. Math.} 233:7 (2010), 1649--1659.




\bibitem{ContiCotroneiSauer15}
{C.~Conti, M.~Cotronei, T. Sauer},
{Factorization of Hermite subdivision operators preserving
exponentials and polynomials}, Adv. Comput. Math. 45 (2017), 1055--1079.

\bibitem{ContiCotroneiSauer16}
{C.~Conti, M.~Cotronei, T. Sauer},
{Convergence of level dependent Hermite subdivision schemes}, submitted.

\bibitem{CMR}
C.~Conti, J.L.~Merrien, L.~Romani:
Dual Hermite Subdivision Schemes of de Rham-type.
BIT Numerical Mathematics {54}(4) (2014), 955--977.

\bibitem{CRU}
C.~Conti,L.~Romani, M.~Unser: Ellipse-Preserving Hermite interpolation and Subdivision. J. Math. Anal. Appl. 426 (2015), 211--227.



%

 


\bibitem{CH} M. Cotronei,  M. Holschneider, {Partial parameterization of orthogonal wavelet matrix filters},
{J. Comput. Appl. Math.} 243 (2013), 113--125.

\bibitem{CLS} M. Cotronei, L. Lo Cascio, T. Sauer, \emph{Multifilters and prefilters: Uniqueness and algorithmic aspects},
{J. Comput. Appl. Math.} 221 (2008), 346--354.

\bibitem{DubucMerrien06}
S.~Dubuc and J.-L. Merrien, {Convergent vector and {H}ermite
 subdivision schemes}, Constr. Approx. 23 (2006), 1--22.

\bibitem{DynLevin99}
N. Dyn, D. Levin: Analysis of Hermite-interpolatory subdivision schemes. In: Dubuc, S., Deslauriers, G. (eds.) Spline Functions and the Theory of Wavelets,   American Mathematical Society, Providence (1999), 105--113.

\bibitem{DynLevin}
N. Dyn, D. Levin, {Subdivision schemes in geometric modelling}, Acta Numerica 11 (2002), 73--144




  
\bibitem{Keinert} F. Keinert, Wavelets and Multiwavelets, Chapman \& Hall/CRC, (2004).

\bibitem{LebrunVetterli} J. Lebrun, M. Vetterli, High-order balanced multiwavelets: theory,
factorization, and design, {IEEE Trans. Signal Process.} 49:9 (2001), 1918-1930.
	
\bibitem{MerrienSauer12}
J.-L. Merrien and T.~Sauer, {From {H}ermite to stationary subdivision
  schemes in one and several variables}, Advances Comput. Math. {36}
  (2012), 547--579.
  
 
  
  
\bibitem{PitolliACOM}
F. Pitolli, Bell-shaped nonstationary refinable ripplets, 
 Adv. Comput. Math. 42 (2016), 1427--1451.

\bibitem{Shumilov}
B.~M. Shumilov and U.~S. Ymanov,
{"Lazy" Wavelets of Hermite Qunitic
Splines and a Splitting Algorithm},
Universal J. of Comput. Math. (2013), 109--117
.

\bibitem{Strela1995}
V.~Strela and G.~Strang,
{ Finite Element Multiwavelets}, in {Approximation Theory, Wavelets and Applications},  Springer Netherlands (1995), 485--496.
  
  \bibitem{UnserBlu05}
  M.~Unser and T.~Blu, {Cardinal Exponential Splines: Part I -- 	Theory and Filtering Algorithms}, IEEE
  Trans. Sig. Proc. {53} (2005), 1425--1438.
  
\end{thebibliography}
\end{document}